\documentclass{amsart}
\usepackage{amsfonts, color}
\usepackage{latexsym}
\usepackage{amssymb}
\usepackage{amsmath}

\newcommand{\R}{\mathbb R}
\newcommand{\N}{\mathbb N}

\newcommand{\E}{\mathbb E}
\newcommand{\Pro}{\mathbb P}

\newtheorem{thm}{Theorem}[section]
\newtheorem{cor}{Corollary}[section]
\newtheorem{lemma}{Lemma}[section]

\newtheorem{proposition}{Proposition}[section]
\newtheorem{conjecture}{Conjecture}[section]
\theoremstyle{remark}
\newtheorem*{rmk}{Remark}
\DeclareMathOperator{\signum}{sgn}

\begin{document}


\title{The variance conjecture on hyperplane projections of the $\ell_p^n$
balls}

\author[D.\,Alonso]{David Alonso-Guti\'errez}
\address{\'Area de an\'alisis matem\'atico, Departamento de matem\'aticas, Facultad de Ciencias, Universidad de Zaragoza, Pedro cerbuna 12, 50009 Zaragoza (Spain), IUMA}
\email{alonsod@unizar.es}

\author[J.\,Bastero]{Jes\'us Bastero}
\address{\'Area de an\'alisis matem\'atico, Departamento de matem\'aticas, Facultad de Ciencias, Universidad de Zaragoza, Pedro cerbuna 12, 50009 Zaragoza (Spain), IUMA}
\email[(Jes\'us
Bastero)]{bastero@unizar.es}
\subjclass[2010]{Primary 52B09, Secondary 52A23}
\thanks{Partially supported by Spanish grants MTM2013-42105-P, DGA E-64, and P1·1B2014-35
projects}
\begin{abstract}
We show that for any $1\leq p\leq\infty$, the family of random vectors uniformly distributed on hyperplane projections
of the unit ball of $\ell_p^n$ verify the variance conjecture
$$
\textrm{Var}\,|X|^2\leq C\max_{\xi\in S^{n-1}}\E\langle X,\xi\rangle^2\E|X|^2,
$$
where $C$ depends on $p$ but not on
the dimension $n$ or the hyperplane. We will also show a general result relating the variance conjecture for a random vector uniformly distributed on an isotropic convex body and the variance
conjecture for a random vector uniformly distributed on any Steiner symmetrization of it. As a consequence we will have that the class of random vectors uniformly distributed on any Steiner symmetrization of an $\ell_p^n$-ball verify the variance conjecture.
\end{abstract}

\date{\today}
\maketitle
\section{Introduction and notations}

A probability measure $\mu$ on $\R^n$ is said to be log-concave if it has a density with
respect to the Lebesgue measure
$$
d\mu(x)=e^{-V(x)}dx,
$$
where $V:\R^n\to(-\infty,\infty]$ is a convex function. For instance, the
uniform probability measure on a convex body and the Gaussian measure are
examples of log-concave probabilities on $\R^n$. A log-concave random vector $X$ is a random vector in $\R^n$ distributed according to a log-concave probability measure. A log-concave random vector $X$ is called
isotropic if the following two conditions hold:
\begin{itemize}
\item The barycenter is at the origin, i.e., $\E X=0$,
\item The covariance matrix is the identity $I_n$, i.e. $\E\langle
X,e_i\rangle\langle X,e_j\rangle=\delta_{i,j}$,
\end{itemize}
where $\{e_i\}_{i=1}^n$ denotes the canonical basis in $\R^n$, $\delta_{i,j}$ is
the Kronecker delta, $\langle\cdot,\cdot\rangle$ is the usual scalar product in
$\R^n$, and $\E$
denotes the expectation. We will also denote by $\textrm{Var}$ the variance. It is well known that for any log-concave random vector $X$
there exists an affine map $T$, with non-zero determinant, such that $TX$ is isotropic. If $X$ is centered then $T$ is non-degenerate linear map $T\in GL(n)$.

Given a centered log-concave random vector $X$, we will denote by
$\lambda_X^2$ the largest eigenvalue of its covariance matrix $M_X$
$$
\lambda_X^2=\Vert M_X\Vert_{\ell_2^n\to\ell_2^n}=\max_{\xi\in
S^{n-1}}\E\,\langle X,\xi\rangle^2,
$$
where $S^{n-1}$ denotes the Euclidean unit sphere in $\R^n$.

The variance conjecture was considered by Bobkov and Koldobsky in the context of
the Central Limit Problem for isotropic convex bodies (see \cite{BK}) and it
states the following:
\begin{conjecture}\label{VarianceIsotropic}
There exists an absolute constant $C$ such that for every isotropic log-concave
random vector $X$
$$
\textrm{Var}\,|X|^2\leq C\E\,|X|^2=Cn.
$$
\end{conjecture}

It was conjectured before by Antilla, Ball, and Perissinaki (see \cite{ABP})
that for an isotropic log-concave random vector $X$, $|X|$ is highly
concentrated in a ``thin shell'' more than the trivial bound
$\textrm{Var}\,|X|\leq\E\,|X|^2$ suggests.

The variance conjecture is a particular case of a stronger conjecture, due to
Kannan, Lov\'asz, and Simonovits (see \cite{KLS}), concerning the spectral gap of
log-concave probability measures. This conjecture can be stated in the following way
due to the work of Cheeger, Maz'ya and Ledoux, among others:
\begin{conjecture}\label{KLSConjecture}
There exists an absolute constant $C$ such that
for any centered log-concave random vector $X$ and for any locally Lipschitz
function $g:\R^n\to\R$ such that the random variable $g(X)$ has finite variance
$$
\textrm{Var}\,g(X)\leq C\lambda_X^2\E\,|\nabla g(X)|^2.
$$
\end{conjecture}
 Notice that Conjecture \ref{VarianceIsotropic} is the particular case of
Conjecture \ref{KLSConjecture} when $g(X)=|X|^2$ and $X$ is isotropic. One can
also consider the particular case in which $g(X)=|X|^2$ but $X$ is not
necessarily isotropic. This gives the following general variance conjecture

\begin{conjecture}\label{VarianceGeneral}
There exists an absolute constant $C$ such that for every centered log-concave
random vector $X$
$$
\textrm{Var}\,|X|^2\leq C\lambda_X^2\E\,|X|^2.
$$
\end{conjecture}
This general variance conjecture was considered before in \cite{AB1}, where it
was shown that uniform probability measures on hyperplane projections of $B_1^n$
and $B_\infty^n$ (the unit balls of $\ell_1^n$ and $\ell_\infty^n$) verify it. In the particular case that we consider $X$ isotropic this conjecture becomes  Conjecture \ref{VarianceIsotropic}. However, it is not clear whether these conjectures are equivalent since the general case is not deduced from the isotropic case because we are considering only the function $g(X)=|X|^2$. Some estimates for the constant in Conjecture \ref{VarianceGeneral}, when considering linear deformations of isotropic random vectors verifying Conjecture \ref{VarianceIsotropic} were given in \cite{AB1} and \cite{AB2}.

Not many examples are known to verify these conjectures.
Conjecture \ref{KLSConjecture} is known to be true for a Gaussian random vector and
random vectors uniformly distributed on
the $\ell_p^n$-balls, some revolution bodies, the simplex, and, with an extra
$\log n$ factor, on unconditional
bodies and log-concave probabilities with many symmetries (see \cite{BaC},
\cite{BaW}, \cite{B}, \cite{H}, \cite{K}, \cite{LW}, \cite{S}).
The best general known result in Conjecture \ref{KLSConjecture} adds a factor
$n^\frac{2}{3}(\log n)^2$ and is due to
Gu\'edon-Milman, who proved the best known estimate in Conjecture
\ref{VarianceIsotropic} with an extra factor
$n^\frac{2}{3}$ (see \cite{GM}), and Eldan, who proved that the variance
conjecture implies the Kannan-Lov\'asz-Simonovits
conjecture, up to a logarithmic factor (see \cite{E}). Besides, Conjecture
\ref{VarianceGeneral} (and thus, \ref{VarianceIsotropic})
is true for random vectors uniformly distributed on unconditional bodies \cite{K} and, as mentioned before,
hyperplane projections of $B_1^n$ and
$B_\infty^n$ (see \cite{AB1}), and increments of log-concave martingales (see \cite{CG}). For more information on these conjectures and
their relation with some other problems
in asymptotic convex geometry we also refer the reader to the monographs
\cite{BGVV}  and
\cite{AB2}.

In this paper we approach the study of the general variance conjecture for the
class of random vectors uniformly distributed on projections of $B_p^n$, the unit balls
of $\ell_p^n$ $1<p<\infty$, onto $(n-1)$-dimensional subspaces $H=\theta^\perp$,
extending the results obtained for $p=1,\infty$ in \cite{AB1}. Namely, we will
prove the following
\begin{thm}\label{projectionsB_p^n}
There exists an absolute constant $C$ such that for any hyperplane
$H=\theta^\perp$, with $\theta\in S^{n-1}$, if $X$ is a random vector uniformly distributed on $P_H(B_p^n)$ we have that if $p\leq n$
$$
\textrm{Var}\,|X|^2\leq C\log(1+p)\lambda_X^2\E\,|X|^2
$$
and if $p>n$
$$
\textrm{Var}\,|X|^2\leq C\lambda_X^2\E\,|X|^2.
$$

Furthermore, if $1\leq p\leq n$ the set of vectors $\theta\in S^{n-1}$ such that
$$
\textrm{Var}\,|X|^2\leq C\lambda_X^2\E\,|X|^2
$$
has Haar probability measure greater than $1-\frac{1}{2^n}$.
\end{thm}

Notice that the value of the constant in the theorem depends on $p$ if $p\leq n$ and does not depend on $p$ if $p>n$.  The reason for this discontinuity in the value of the constant is just technical. Our proof gives a constant $C\log(1+p)$ for every value of $p\in[1,\infty]$ and, using a different method we were able to give a better estimate, independent of $p$, that holds for values of $p$ greater than $n$.



We would like to remark that we are considering a random vectors uniformly distributed on
projections of $B_p^n$ and not the projections of random vectors uniformly distributed on
$B_p^n$. When considering the projections of the random vectors the situation is much simpler. Even though it is probably straightforward for specialists, for the sake of completeness  we will give in Section \ref{proofSteiner} a general
result showing that an isotropic log-concave random vector verifies the variance conjecture if and only if any of its hyperplane projections does.

A convex body $K$ is called isotropic if it has volume 1, $|K|=1$, and for any
vector $\theta\in S^{n-1}$ we have $\E\langle X,\theta\rangle=0$ and $\E\langle
X,\theta\rangle^2=L_K^2$, where $X$ is a random vector uniformly distributed on $K$
and $L_K$ does not depend on $\theta$ and is called the isotropic constant of
$K$. Thus, $K$ is isotropic if and only if the random vector uniformly distributed on
$L_K^{-1}K$ is isotropic. Given a convex body $K$ and a hyperplane
$H=\theta^\perp$, with $\theta\in S^{n-1}$, the Steiner symmetrization of $K$
with respect to $H$ is the convex body defined as
$$
S_\theta(K)=\left\{x+t\theta\,:\,x\in P_{\theta^\perp}K, |t|\leq
\frac{1}{2}|K\cap \left(x+\langle\theta\rangle\right)|\right\},
$$
where $\langle\theta\rangle$ denotes the one-dimensional subspace spanned by
$\theta$. We
will also study the relation between the variance conjecture for a random vector uniformly distributed on an isotropic convex body and a random vector uniformly distributed on the Steiner symmetrization of it with respect to any hyperplane. We
will  show the following general result, which shows that a random vector uniformly distributed on an isotropic body verifies the variance conjecture if and only if
a random vector uniformly distributed on any of its Steiner symmetrizations does. As a consequence, if a random vector uniformly distributed on an isotropic convex body $K$ verifies the variance conjecture, then the class of random vectors uniformly distributed on any of its Steiner symmetrizations also verify the variance conjecture.

\begin{thm}\label{SteinerSymmetrization}
Let $K$ be an isotropic convex body and $\theta\in S^{n-1}$. Let us denote by
$X$ a random vector uniformly distributed on $K$ and by $Y_\theta$ a random vector uniformly distributed on $S_\theta(K)$, the Steiner symmetrization of $K$ with
respect to $H=\theta^\perp$. Then the following are equivalent
\begin{itemize}
\item There exists a constant $C_1$ such that
$$
\textrm{Var}\,|X|^2\leq C_1\lambda_X^2\E\,|X|^2.
$$
\item There exists a constant $C_2$ such that
$$
\textrm{Var}|Y_\theta|^2\leq
C_2\lambda_{Y_\theta}^2\E\,|Y_\theta|^2
$$
for some $\theta\in S^{n-1}$.
\item There exists a constant $C_3$ such that
$$
\textrm{Var}\,|Y_\theta|^2\leq C_3\lambda_{Y_\theta}^2\E\,|Y_\theta|^2
$$
for every $\theta\in S^{n-1}$,
\end{itemize}
where
$$
C_2\leq C_3\leq 2(C_1+ C)\textrm{ and }C_1\leq C_2+ C,
$$
with $C$ an absolute constant.
\end{thm}

The paper is organized as follows: We will prove Theorem \ref{projectionsB_p^n} in Section \ref{proofProjections}. In Section \ref{Preliminaries} we will
present some known results that we will use and in Section \ref{ProbabilisticEstimates} we will prove some technical lemmas we will need to prove Theorem
\ref{projectionsB_p^n}. Finally, in Section \ref{proofSteiner} we will show the
general results concerning the variance conjecture for projections of isotropic log-concave random vectors and  for random vectors uniformly distributed on the Steiner symmetrizations of an isotropic convex body. We will always use the letters $c, C, C^\prime$ to
denote absolute constants and will use $a\sim b$ to denote the existence of two
positive absolute constants $c, C$ such that $ca\leq b\leq Ca$.

\section{Preliminaries}\label{Preliminaries}
In this Section we present the tools we use to prove the aforementioned results.
We will use the techniques developed in \cite{BaN}. We will denote by
$\sigma_p^n$ the uniform area measure (Hausdorff measure) on $\partial B_p^n$, the boundary of
$B_p^n$, and by $\mu_p^n$
the cone probability measure on $\partial B_p^n$, defined by
$$
\mu^n_p(A)=\frac {|\{ta\in\R^n;a\in A, 0\leq t\leq 1   \}| }
{|B^n_p|}  \qquad A\subseteq \partial B^n_p.
$$
A relation between these two measures was proved in \cite{NR}. For the sake of completeness we include a short proof of it in the following lemma:
\begin{lemma}
Let $\sigma_p^n$ and $\mu_p^n$ be the uniform area measure and the cone probability measure on $\partial B_p^n$. Then
$$\frac{d\sigma^n_p(x)}{d\mu_p^n(x)}=n|B^n_p|\left\vert
\nabla(\Vert\cdot\Vert_p)(x)\right\vert
$$
for almost every point $x\in \partial B^n_p$.
\end{lemma}
\begin{proof}
Let $g:\partial B^n_p\to \R$ be an integrable function with respect to $\mu_p^n$. Denoting by $\sigma_{t\partial B^n_p}$ the uniform area measure on $t\partial B^n_p$ and using the co-area formula, we have that
\begin{align*}
 \int_{\partial B_p^n} g(y)d\mu_p^n(y)&=\frac{1}{|B^n_p|}\int_{B^n_p}g\left(\frac{x}{\Vert x\Vert_p}\right)dx\\
&=\frac{1}{|B^n_p|}\int_0^1
\int_{t\partial B^n_p}\dfrac{g\left(\frac{x}{\Vert x\Vert_p}\right)}
{\vert \nabla(\Vert \cdot\Vert_p)(x)\vert}d\sigma_{t\partial B^n_p}(x)dt\\
&=\frac{1}{|B^n_p|}\int_0^1 t^{n-1}
\int_{\partial B^n_p}\dfrac{g(y)}{\vert \nabla(\Vert \cdot\Vert_p )(y)\vert}d\sigma_p^n(y)dt\\
&=\int_{\partial B^n_p}\frac{1}{n|B^n_p|}
\dfrac{g(y)}{\vert \nabla(\Vert \cdot\Vert_p)(y)\vert}d\sigma_p^n(y).\\
\end{align*}
\end{proof}
Consequently, by using Cauchy's
formula,
if $H=\theta^\perp$, $X$ is a random vector uniformly distributed on $K=P_HB_p^n$ and $f:K\to\R$ is a
Borel
integrable function
\begin{eqnarray*}
\E\, f(X)&=&\frac{1}{|K|}\int_Kf(x)dx\\&=&
\frac{1}{2|K|}\int_{\partial B^n_p}f(P_H(y))\frac{\left|\langle
\nabla \Vert \cdot\Vert_p(y),\theta\rangle\right|}{\left|
\nabla\Vert \cdot\Vert_p(y)\right|}d\sigma^n_p(y)\cr
&=&\frac{\int_{\partial B_p^n} f(P_H(y))|\langle
\nabla(\Vert\cdot\Vert_p)(y),\theta\rangle|d\mu_p^n}{\int_{\partial B_p^n}
|\langle \nabla(\Vert\cdot\Vert_p)(y),\theta\rangle|d\mu_p^n}\cr
&=&\frac{\int_{\partial B^n_p} f(P_H(y))\left|\sum_{i=1}^n
|y_i|^{p-1}\signum (y_i) \theta_i\right|d\mu^n_p(y)}{\int_{\partial B^n_p}
\left|\sum_{i=1}^n
|y_i|^{p-1}\signum (y_i) \theta_i\right|d\mu^n_p(y)}.
\end{eqnarray*}

We will use the following probabilistic description of the measure $\mu_p^n$
(see, for instance, \cite{SZ1}, \cite{BaN},
\cite{NR}): Let $g_1,\dots,g_n$ be independent copies of a random
variable $g$ with density with respect to the Lebesgue measure
$$\frac{e^{-|t|^p}}{2\Gamma(1+1/p)}$$ for every $t\in\R$ and denote by
$$S=\left(\sum_{i=1}^n|g_i|^p\right)^\frac{1}{p}.$$ Then
\begin{itemize}
\item The random vector
$\dfrac GS:=\left(\dfrac{g_1}{S},\dots,\dfrac{g_n}{S}\right)$  and the random
variable $S$ are independent.
\item $\dfrac GS$ is distributed on $\partial B_p^n$ according to the cone
measure $\mu_p^n$.
\end{itemize}
Hence
$$
\E\,f(X)=\frac{\E
f\left(P_H\left(\frac{g_1}{S},\dots,\frac{g_n}{S}\right)\right)\left|\sum_{i=1}
^n
\frac{|g_i|^{p-1}}{S^{p-1}}\signum (g_i) \theta_i\right|}{\E
\left|\sum_{i=1}^n
\frac{|g_i|^{p-1}}{S^{p-1}}\signum (g_i) \theta_i\right|}.
$$

By the independence of $\frac{G}{S}$ and $S$, we have
\begin{eqnarray*}\E\,f(X)
&=&\frac{\E
f\left(P_H\left(\frac{g_1}{S},\dots,\frac{g_n}{S}\right)\right)\left|\sum_{i=1}
^n
|g_i|^{p-1}\signum (g_i) \theta_i\right|}{\E
\left|\sum_{i=1}^n
|g_i|^{p-1}\signum (g_i) \theta_i\right|}\\
&=&\frac{\E
f\left(P_H\left(\frac{G}{S}\right)\right)\psi_\theta}{\E
\psi_\theta},
\end{eqnarray*}where $\psi_\theta$ is defined as
\begin{equation}\label{Psi}
 \psi_\theta=\left|\sum_{i=1}^n
|g_i|^{p-1}\signum (g_i) \theta_i\right|.
\end{equation} We will sometimes use the notation $\psi$ instead of $\psi_\theta$ when there is no possibility of confusion.

The following theorem, which will be used to obtain some estimates for the expected value of $\psi$,  was proved in \cite{ACPP}:
\begin{thm} \label{thm:orlicz_p_norm}
Let $1<q<\infty$, $X_1,\ldots,X_n$ be independent identically distributed integrable random variables. For every $s\geq0$ define
\[
M(s) = \frac{q}{q-1}\int_0^s\left( \int_{|X_1| \leq \frac{1}{t}} t^{q-1} \left| X_1 \right|^q d \Pro + \int_{ |X_1| > 1/t}|X_1|d\mathbb P \right)dt .
\]
Then, for every $x\in\R^n$,
\[
 c_1 (q-1)^{1/q} \Vert x \Vert_M \leq \mathbb E \left(\sum_{i=1}^n|x_iX_i|^q\right)^{\frac{1}{q}} \leq c_2 \Vert x \Vert_M,
\]
where $c_1,c_2,$ are positive absolute constants and $\Vert x\Vert_M$ denotes the Luxemburg norm given by the Orlicz function $M$, which is defined by
$$
\Vert x \Vert_M=\inf\left\{\rho>0\,:\,\sum_{i=1}^nM\left(\frac{|x_i|}{\rho}\right)\leq1\right\}.
$$
\end{thm}

We will also make use of the following theorem, which was proved in \cite{KS}:

\begin{thm}\label{AveragePermutations}
Let $1\leq q\leq\infty$ and $a\in\R^{n\times n}$. Then
$$
\textrm{Ave}_\pi\left(\sum_{i=1}^n|a_{i,\pi(i)}|^q\right)^\frac{1}{q}\sim \frac{1}{n}\sum_{k=1}^n (a_{i,j}^*)_k+\left(\frac{1}{n}\sum_{k=n+1}^{n^2}(a_{i,j}^*)_k^q\right)^\frac{1}{q},
$$
where $a_{i,j}^*\in\R^{n^2}$ is the decreasing rearrangement of $a$ and $\pi$ runs over all the permutations of $\{1,\dots, n\}$.
\end{thm}

In the same paper the authors showed that when $q=2$ this estimate can be estimated by using an Orlicz function.
%

\section{Some probabilistic estimates}\label{ProbabilisticEstimates}
In this section we will prove several technical lemmas we will need in order to prove Theorem \ref{projectionsB_p^n}. The following lemma is well known:

\begin{lemma}\label{LemmagS}
Let $\alpha\geq 0$ and let $g_1,\dots,g_n$ be independent copies of a random variable
$g$, with density with respect to the Lebesgue measure $\frac{e^{-|t|^p}}{2\Gamma(1+1/p)}$, and $S=\left(\sum_{i=1}^n|g_i|^p\right)^\frac{1}{p}$. Then
$$
\E|g|^\alpha=\frac{\Gamma\left(\frac{\alpha+1}{p}\right)}{\Gamma\left(\frac{1}{p
}\right)}
$$
and
$$
\E S^\alpha=\frac{\Gamma\left(
 \frac{n+\alpha}{p}\right)}{\Gamma\left(
 \frac{n}{p}\right)}
$$
\end{lemma}
\begin{proof}
The value of $\E|g|^\alpha$ can be computed directly. Let us compute $\E S^\alpha$.
\[
  \E S^\alpha=\E\left(\sum_{i=1}^n|g_i|^p\right)^{\alpha/p}=\int_{\R^n}
 \Vert x\Vert_p^{\alpha}\frac{e^{-\Vert
x\Vert_p^p}}{\left(2\Gamma(1+1/p)\right)^n}dx.
\]
Changing to polar coordinates
\begin{align*}
 \E S^\alpha&=\frac{n|B^n_p| }{\left(2\Gamma(1+1/p)\right)^n}
 \int_0^\infty r^{n+\alpha-1}
e^{-r^p}dr\end{align*}
and this expression implies the result.
\end{proof}

This lemma implies the following:
\begin{lemma}\label{MomentsSumsSymmetric}
Let $X_1,\dots, X_n$ be independent copies of $X=g^2-\bar{g}^2$, where $\bar{g}$ is an independent copy of $g$, defined as before. Then, for any $2\leq\alpha\leq e^p$ we have
$$
\left(\E\left|\sum_{i=1}^n X_i\right|^\alpha\right)^\frac{1}{\alpha}\leq C\sqrt{\alpha n}.
$$
\end{lemma}
\begin{proof}
By the triangle inequality
$$
\left(\E|X|^\alpha\right)^\frac{1}{\alpha}\leq2\left(\E|g|^{2\alpha}\right)^\frac{1}{\alpha}=2\left(\frac{\Gamma\left(\frac{1+2\alpha}{p}\right)}{\Gamma\left(\frac{1}{p}\right)}\right)^\frac{1}{\alpha}.
$$
Using Stirling's formula
$$
\left(\E|X|^\alpha\right)^\frac{1}{\alpha}\leq C \alpha^\frac{2}{p}\leq C_1,
$$
since $\alpha\leq e^p$. Now, since the random variables $X_i$ are symmetric, taking $\varepsilon_1,\dots,\varepsilon_n$ independent Bernoulli random variables, which are also independent of the random variables $X_i$, we have
$$
\E\left|\sum_{i=1}^n X_i\right|^\alpha=\E\E_{\varepsilon}\left|\sum_{i=1}^n \varepsilon_iX_i\right|^\alpha
$$
and, by Khintchine's inequality (see \cite{HA} for the best value of the constant in Khintchine's inequality)
$$
\E_{\varepsilon}\left|\sum_{i=1}^n \varepsilon_iX_i\right|^\alpha\leq (C_2\sqrt{\alpha})^\alpha\left(\sum_{i=1}^n|X_i|^2\right)^\frac{\alpha}{2}\leq (C_2\sqrt{\alpha})^\alpha n^{\frac{\alpha}{2}-1}\sum_{i=1}^n|X_i|^\alpha.
$$
Hence
$$
\left(\E\left|\sum_{i=1}^n X_i\right|^\alpha\right)^\frac{1}{\alpha}\leq C_1C_2\sqrt{\alpha n}.
$$
\end{proof}
Let us recall that for every $\theta\in S^{n-1}$, $\psi_\theta$ was defined like
$$
\psi_\theta=\left|\sum_{i=1}^n|g_i|^{p-1}\signum(g_i)\theta_i\right|.
$$
We will also call
$$
\phi_\theta=\left(\sum_{i=1}^n|g_i|^{2p-2}\theta_i^2\right)^\frac{1}{2}.
$$
Notice that since the random variables $g_i$ are symmetric with respect to the origin, for
any
choice of signs $\varepsilon_i=\pm 1$ we have
$$
\E\psi_\theta=\E\left|\sum_{i=1}^n|g_i|^{p-1}\signum(g_i)\theta_i\right|=\E\left|\sum_{
i=1}
^n|\varepsilon_ig_i|^{p-1}\signum(\varepsilon_ig_i)\theta_i\right|.
$$
Thus, taking $\varepsilon_1,\dots,\varepsilon_n$ independent Bernoulli random
variables, by Khintchine's inequality we have
\begin{eqnarray*}
\E\psi_\theta&=&\E_\varepsilon\E_g\left|\sum_{i=1}
^n|\varepsilon_ig_i|^{p-1}\signum(\varepsilon_ig_i)\theta_i\right|=\E_g\E_\varepsilon\left|\sum_{i=1}^n\varepsilon_i
|g_i|^{p-1}\signum(g_i)\theta_i\right|\cr
&\sim&\E\left(\sum_{i=1}^n|g_i|^{2p-2}\theta_i^2\right)^\frac{1}{2}=\E\phi_\theta.
\end{eqnarray*}
The following lemma gives estimates for the value of $\E\psi_\theta$, independent of the direction $\theta$, in terms of the $\Vert\theta\Vert_1$, or in terms of the value of $\E\psi_{\theta_0}$, where $\theta_0$ is the diagonal direction.
\begin{lemma}\label{ExpectationPsi}
Let $\theta_0=\left(\frac{1}{\sqrt n},\dots,\frac{1}{\sqrt n}\right)$. Then
\begin{itemize}
\item[a)]There exist absolute constants $C_1,C_2$ such that for any $1\leq p<\infty$ and
$\theta\in S^{n-1}$
$$
\frac{C_1}{p}\leq\E\,\psi_\theta\leq\frac{C_2}{\sqrt p}.
$$
Furthermore, for any $1\leq p<\infty$ and
$\theta\in S^{n-1}$
$$
\frac{C_1}{p}\leq\E\,\psi_\theta^2\leq\frac{C_2}{p}.
$$
\item[b)]There exist two absolute constants $C_1,C_2$ such that for any $1\leq p<\infty$
$$
\frac{C_1}{\sqrt n}\E\psi_{\theta_0}\Vert\theta\Vert_1\leq\E\psi_\theta\leq\frac{C_2}{p}\Vert\theta\Vert_1.
$$
\item[c)] There exists an absolute constant $C$ such that
$$
\E\psi_\theta\leq C\E\psi_{\theta_0}.
$$
Furthermore, there exists an absolute constant $c$ such that
$$
\sigma\left\{\theta\in S^{n-1}\,:\,\E\psi_\theta\geq c\E\psi_{\theta_0}\right\}\geq 1-\frac{1}{2^n}.
$$
\end{itemize}
\end{lemma}
\begin{proof}
Let us first prove a):

By Jensen's inequality we have
\begin{eqnarray*}
\E\psi_\theta&\sim&\E\phi_\theta\geq\E\sum_{i=1}^n|g_i|^{p-1}\theta_i^2=\E|g|^{p-1}
=\frac{1}{\Gamma\left(\frac{1}{p}\right)}=\frac{1}{p}\frac{1}{
\Gamma\left(1+\frac{1}{p}\right)}\sim\frac{1}{p}.
\end{eqnarray*}
On the other hand, by H\"older's inequality
\begin{eqnarray*}
 \E\psi_\theta&\sim&\E\phi_\theta\leq\left(\E\sum_{i=1}^n|g_i|^{2p-2}\theta_i^2\right)^\frac{1}{2}
=\left(\E|g|^{2p-2}\right)^\frac{1}{2}=\left(\frac{\Gamma\left(\frac{2p-1}{p}
\right)}{\Gamma\left(\frac{1}{p}\right)}\right)^\frac{1}{2}\cr
&=&\left(\frac{1}{p}\frac{\Gamma\left(\frac{2p-1}{p}\right)}{\Gamma\left(1+\frac
{1}{p}\right)}\right)^\frac{1}{2}\sim\frac{1}{\sqrt p}.
\end{eqnarray*}

In the same way, taking independent Bernoulli random
variables and using Khintchine's inequality
\begin{eqnarray*}
\E\psi_\theta^2&\sim&\E\phi_\theta^2=\E|g|^{2p-2}
=\frac{\Gamma\left(\frac{2p-1}{p}\right)}{\Gamma\left(\frac{1}{p}\right)}=\frac{1}{p}\frac{\Gamma\left(\frac{2p-1}{p}\right)}{
\Gamma\left(1+\frac{1}{p}\right)}\sim\frac{1}{p}.
\end{eqnarray*}

Let us now prove b):

Notice that if $p=1$, by Khintchine's inequality $\E\psi_\theta\sim 1$ for every $\theta\in S^{n-1}$ and then the result follows. Assume that $p>1$. On the one hand, by Lemma \ref{LemmagS}
$$
\E\psi_\theta\leq\E\left(\sum_{i=1}^n|g_i|^{p-1}\left|\theta_i\right|\right)\leq\frac{c_2}{p}\Vert\theta\Vert_1.
$$
On the other hand,
$$
\E\psi_\theta\sim\E\phi_\theta=\E\Vert(|g_i|^{p-1}\theta_i)_{i=1}^n\Vert_2.
$$

Thus, applying Theorem \ref{thm:orlicz_p_norm} with $X_i=|g_i|^{p-1}$ and $q=2$, we have that
$$
\E\psi_{\theta}\sim\Vert\theta\Vert_M,
$$
with
\begin{align*}
M(s)&=2\int_0^s\left(\int_0^{t^{-\frac{1}{p-1}}}tx^{2p-2}\frac{e^{-x^p}}{\Gamma\left(1+\frac{1}{p}\right)}dx+\int_{t^{-\frac{1}{p-1}}}^\infty x^{p-1}\frac{e^{-x^p}}{\Gamma\left(1+\frac{1}{p}\right)}dx\right)dt\\
&=\frac{2}{p\Gamma\left(1+\frac{1}{p}\right)}\int_0^s\left(\int_0^{t^{-p^*}}tr^\frac{p-1}{p}e^{-r}dr+\int_{t^{-p^*}}^\infty e^{-r}dr\right)dt,\\
\end{align*}
where $p^*=\frac{p}{p-1}$ is the dual exponent of $p$. Let $B_M$ be the unit ball of $\Vert\cdot\Vert_M$. Taking into account that the norm $\Vert\cdot\Vert_M$ is 1-symmetric we have that
$$
B_M\subseteq\frac{n}{\Vert(1,\dots,1)\Vert_M}B_1^n.
$$
Thus, for any $\theta\in S^{n-1}$
$$
\Vert\theta\Vert_M\geq\frac{\Vert\theta_0\Vert_M}{\sqrt n}\Vert\theta\Vert_1
$$
and so
$$
\E\psi_\theta\geq\frac{c_1}{\sqrt n}\E\psi_{\theta_0}\Vert\theta\Vert_1.
$$

Finally, we prove c):

Since for any permutation $\pi$ of $\{1,\dots,n\}$
$$
\E\psi_\theta\sim\E\phi_\theta=\E\left(\sum_{k=1}^n|g_k|^{2p-2}\theta_k^2\right)^\frac{1}{2}=\E\left(\sum_{k=1}^n|g_k|^{2p-2}\theta_{\pi(k)}^2\right)^\frac{1}{2}
$$
we have that this expectation equals
$$
\E\textrm{Ave}_\pi\left(\sum_{k=1}^n|g_k|^{2p-2}\theta_{\pi(k)}^2\right)^\frac{1}{2}
$$
which, by Theorem \ref{AveragePermutations} applied to $a_{i,j}=|g_i|^{p-1}\theta_j$, is equivalent to
$$
\E\left(\frac{1}{n}\sum_{k=1}^{n}(|g_i|^{p-1}\theta_j)^*_k+\left(\frac{1}{n}\sum_{k=n+1}^{n^2}(|g_i|^{2p-2}\theta_j^2)^*_k\right)^\frac{1}{2}\right).
$$
Now, since by H\"older's inequality
$$
\frac{1}{n}\sum_{k=1}^{n}(|g_i|^{p-1}\theta_j)^*_k\leq\left(\frac{1}{n}\sum_{k=1}^{n}(|g_i|^{2p-2}\theta_j^2)^*_k\right)^\frac{1}{2}
$$
we have that
\begin{eqnarray*}
&&\left(\frac{1}{n}\sum_{k=1}^{n}(|g_i|^{p-1}\theta_j)^*_k+\left(\frac{1}{n}\sum_{k=n+1}^{n^2}(|g_i|^{2p-2}\theta_j^2)^*_k\right)^\frac{1}{2}\right)\cr
&\leq&\left(\left(\frac{1}{n}\sum_{k=1}^{n}(|g_i|^{2p-2}\theta_j^2)^*_k\right)^\frac{1}{2}+\left(\frac{1}{n}\sum_{k=n+1}^{n^2}(|g_i|^{2p-2}\theta_j^2)^*_k\right)^\frac{1}{2}\right)\cr
&\leq&\sqrt2\left(\frac{1}{n}\sum_{i,j=1}^{n}|g_i|^{2p-2}\theta_j^2\right)^\frac{1}{2}=\sqrt2\phi_{\theta_0}\cr
\end{eqnarray*}
and taking expectation and using Khintchine's inequality again we obtain
$$
\E\psi_\theta\leq c_2\E\psi_{\theta_0}.
$$
Besides, by Markov's inequality for any $A\geq 0$
$$
\frac{|B_1^n|}{|B_2^n|}=\int_{S^{n-1}}\frac{1}{\Vert\theta\Vert_1^n}d\sigma(\theta)\geq\frac{1}{A^n}\sigma\{\theta\in S^{n-1}\,:\,\Vert\theta\Vert_1\leq A\}.
$$
Thus, since $\left(\frac{|B_1^n|}{|B_2^n|}\right)^\frac{1}{n}\leq\frac{C}{\sqrt n}$, taking $A=\frac{1}{2C}\sqrt{n}$, we obtain that
$$
\sigma\left\{\theta\in S^{n-1}\,:\,\Vert\theta\Vert_1\leq \frac{1}{2C}\sqrt{n}\right\}\leq\frac{1}{2^n}
$$
and, by part b) in this lemma, there exists an absolute constant $c$ such that
$$
\sigma\{\theta\in S^{n-1}\,:\,\E\psi_\theta\geq c\E\psi_{\theta_0}\}\geq1-\frac{1}{2^n},
$$
which finishes the proof.
\end{proof}

In both parts b) and c) in Lemma \ref{ExpectationPsi} we have related $\E\psi_\theta$ with $\E\psi_{\theta_0}$. In the following lemma we are going to estimate the value of $\E\psi_{\theta_0}$.

\begin{lemma}\label{ExpectationPsiTheta0}
Let  $\theta_0=\left(\frac{1}{\sqrt n},\dots,\frac{1}{\sqrt n}\right)$. Then, if $1 \leq p\leq n$
$$
\E\psi_{\theta_0}\sim\frac{1}{\sqrt p}
$$
and, if $p=n^\gamma$ with $\gamma>1$,
$$
\E\psi_{\theta_0}\sim\frac{\sqrt n}{p}=\frac{1}{p^{1-\frac{1}{2\gamma}}}.
$$
\end{lemma}

\begin{proof}
By Lemma \ref{ExpectationPsi} $\E\psi_{\theta_0}\leq\frac{C}{\sqrt p}$. Let us prove $\E\psi_{\theta_0}\geq\frac{c}{\sqrt p}$. We have seen that, by Khintchine's inequality,
$$
\E\psi_{\theta_0}\sim\frac{1}{\sqrt n}\Vert(|g_i|^{p-1})_{i=1}^n\Vert_2.
$$
Thus, applying Theorem \ref{thm:orlicz_p_norm} with $X_i=|g_i|^{p-1}$ and $q=2$, we have that
$$
\E\psi_{\theta_0}\sim\frac{1}{\sqrt n}\Vert(1,\dots,1)\Vert_M,
$$
with
\begin{eqnarray*}
M(s)&=&2\int_0^s\left(\int_0^{t^{-\frac{1}{p-1}}}tx^{2p-2}\frac{e^{-x^p}}{\Gamma\left(1+\frac{1}{p}\right)}dx+\int_{t^{-\frac{1}{p-1}}}^\infty x^{p-1}\frac{e^{-x^p}}{\Gamma\left(1+\frac{1}{p}\right)}dx\right)dt\cr
&=&\frac{2}{p\Gamma\left(1+\frac{1}{p}\right)}\int_0^s\left(\int_0^{t^{-p^*}}tr^\frac{1}{p^*}e^{-r}dr+\int_{t^{-p^*}}^\infty e^{-r}dr\right)dt\cr
&=&\frac{2\left(1-\frac{1}{p}\right)}{p\Gamma\left(1+\frac{1}{p}\right)}\int_0^st\int_0^{t^{-p^*}}r^{-\frac{1}{p}}e^{-r}drdt,\cr
\end{eqnarray*}
where the last identity follows from integration by parts and $p^*=\frac{p}{p-1}$ is the dual exponent of $p$.

On the one hand, since
\begin{eqnarray*}
M(s)&\geq&\frac{2}{p\Gamma\left(1+\frac{1}{p}\right)}\int_{\frac{s}{2}}^s\int_0^{t^{-p^*}}tr^\frac{1}{p^*}e^{-r}drdt\cr
&\geq&\frac{s}{p\Gamma\left(1+\frac{1}{p}\right)}\int_{\frac{s}{2}}^s\int_0^{s^{-p^*}}r^\frac{1}{p^*}e^{-r}drdt\cr
&=&\frac{s^2}{2p\Gamma\left(1+\frac{1}{p}\right)}\left(\Gamma\left(2-\frac{1}{p}\right)-\int_{s^{-p^*}}^\infty r^\frac{1}{p^*}e^{-r}dr\right),\cr
\end{eqnarray*}
we have that if $\rho=c\sqrt{\frac{n}{p}}$ and $p\leq \frac{c^2}{\alpha^{\frac{2}{p^*}}}n$ with $\alpha\geq 1$
\begin{eqnarray*}
M\left(\frac{1}{\rho}\right)&\geq&\frac{1}{c^2n\Gamma\left(1+\frac{1}{p}\right)}\left(\Gamma\left(2-\frac{1}{p}\right)-\int_{\rho^{p^*}}^\infty r^\frac{1}{p^*}e^{-r}dr\right)\cr
&\geq&\frac{1}{c^2n\Gamma\left(1+\frac{1}{p}\right)}\left(\Gamma\left(2-\frac{1}{p}\right)-\int_\alpha^\infty r^\frac{1}{p^*}e^{-r}dr\right)\cr
&\geq&\frac{1}{c^2n\Gamma\left(1+\frac{1}{p}\right)}\left(\Gamma\left(2-\frac{1}{p}\right)-\int_\alpha^\infty re^{-r}dr\right)\cr
&=&\frac{1}{c^2n\Gamma\left(1+\frac{1}{p}\right)}\left(\Gamma\left(2-\frac{1}{p}\right)-(\alpha+1)e^{-\alpha}\right)\cr
\end{eqnarray*}
Taking $\alpha$ a constant big enough and then $c$ a constant small enough we have that if $p\leq Cn$ for some absolute constant $C<1$,
$$
M\left(\frac{1}{c\sqrt{\frac{n}{p}}}\right)\geq\frac{1}{n}
$$
and so
$$
\Vert(1,\dots,1)\Vert_M\geq c\sqrt{\frac{n}{p}}.
$$
Consequently,
$$
\E\psi_{\theta_0}\geq\frac{c}{\sqrt p}.
$$

On the other hand, since
\begin{eqnarray*}
M(s)&=&\frac{2\left(1-\frac{1}{p}\right)}{p\Gamma\left(1+\frac{1}{p}\right)}\int_0^st\int_0^{t^{-p^*}}r^{-\frac{1}{p}}e^{-r}drdt\cr
&\geq&\frac{2\left(1-\frac{1}{p}\right)}{p\Gamma\left(1+\frac{1}{p}\right)}\int_0^st\int_0^{t^{-p^*}}t^{\frac{1}{p-1}}e^{-r}drdt\cr
&=&\frac{2\left(1-\frac{1}{p}\right)}{p\Gamma\left(1+\frac{1}{p}\right)}\int_0^st^{1+\frac{1}{p-1}}\left(1-e^{-t^{-p^*}}\right)dt\cr
&\geq&\frac{2\left(1-\frac{1}{p}\right)}{p\Gamma\left(1+\frac{1}{p}\right)}\int_0^st^{1+\frac{1}{p-1}}\left(1-e^{-s^{-p^*}}\right)dt\cr
&=&\frac{2\left(1-\frac{1}{p}\right)s^{2+\frac{1}{p-1}}}{p\left(2+\frac{1}{p-1}\right)\Gamma\left(1+\frac{1}{p}\right)}\left(1-e^{-s^{-p^*}}\right).\cr
\end{eqnarray*}
we have that if $Cn\leq p\leq n$, $\rho=\alpha\sqrt{\frac{n}{p}}$ with $\alpha\leq 1$, there is an absolute constant $c$ such that
\begin{eqnarray*}
M\left(\frac{1}{\rho}\right)&\geq&\frac{cp^\frac{1}{2p-2}}{\alpha^{2+\frac{1}{p-1}}n^{1+\frac{1}{2p-2}}}\left(1-e^{-\alpha^{p^*}}\right)\cr
&\geq&\frac{cp^\frac{1}{2p-2}}{\alpha^{2+\frac{1}{p-1}}n^{1+\frac{1}{2p-2}}}\left(1-e^{-\alpha^{p^*}}\right)\cr
&\geq&\frac{c}{\alpha n},
\end{eqnarray*}
since $p\sim n$. If we take $\alpha$ a constant small enough,
$$
M\left(\frac{1}{\alpha\sqrt\frac{n}{p}}\right)\geq\frac{1}{n}
$$
and so
$$
\Vert(1,\dots,1)\Vert_M\geq \alpha\sqrt\frac{n}{p}.
$$
Consequently,
$$
\E\psi_{\theta_0}\geq \frac{c}{\sqrt p}
$$
also if $Cn\leq p\leq n$.

By Lemma \ref{ExpectationPsi},
$$
\E\psi_{\theta_0}\leq\frac{c_2}{p}\Vert\theta_0\Vert_1= \frac{c_2\sqrt n}{p}
$$
Consequently, if $p=n^\gamma$ with $\gamma>1$,
$$
\E\psi_{\theta_0}\leq\frac{c_2\sqrt n}{p}=\frac{c_2}{p^{1-\frac{1}{2\gamma}}}.
$$
On the other hand, since $p\geq n$, if $n\geq2$
\begin{eqnarray*}
M(s)&=&\frac{2\left(1-\frac{1}{p}\right)}{p\Gamma\left(1+\frac{1}{p}\right)}\int_0^st\int_0^{t^{-p^*}}r^{-\frac{1}{p}}e^{-r}drdt\cr
&\geq&\frac{2\left(1-\frac{1}{p}\right)}{p\Gamma\left(1+\frac{1}{p}\right)}\int_0^st\int_0^{t^{-p^*}}r^{-\frac{1}{p}}e^{-t^{-p^*}}drdt\cr
&=&\frac{2}{p\Gamma\left(1+\frac{1}{p}\right)}\int_0^se^{-t^{-p^*}}dt\cr
&\geq&\frac{2}{p\Gamma\left(1+\frac{1}{p}\right)}\int_{s2^{-\frac{1}{p^*}}}^se^{-t^{-p^*}}dt\cr
&\geq&\frac{2s}{p\Gamma\left(1+\frac{1}{p}\right)}\left(1-2^{-\frac{1}{p^*}}\right)e^{-2s^{-p^*}}\cr
&\geq&\frac{2s}{p\Gamma\left(1+\frac{1}{p}\right)}\left(1-2^{-\frac{n-1}{n}}\right)e^{-2s^{-p^*}}\cr
&\geq&\frac{\sqrt{2}(\sqrt{2}-1)s}{p\Gamma\left(1+\frac{1}{p}\right)}e^{-2s^{-p^*}}.\cr
\end{eqnarray*}
and then
$$
M\left(\frac{1}{s}\right)\geq\frac{\sqrt{2}(\sqrt{2}-1)}{ps}e^{-2s^{p^*}}.
$$
Thus, if $p=n^\gamma$ and we take $s=\alpha n^{1-\gamma}$, with $\alpha\leq 1$
\begin{eqnarray*}
M\left(\frac{1}{\alpha n^{1-\gamma}}\right)&\geq&\frac{\sqrt{2}(\sqrt{2}-1)}{n}e^{-2\alpha^\frac{n^\gamma}{n^\gamma-1} n^{\frac{(1-\gamma)n^\gamma}{n^\gamma-1}}}\cr
&\geq&\frac{\sqrt{2}(\sqrt{2}-1)}{n}e^{-2\alpha^\frac{n^\gamma}{n^\gamma-1}}\cr
&\geq&\frac{\sqrt{2}(\sqrt{2}-1)}{n}e^{-2\alpha}\cr
&\geq&\frac{1}{n}\cr
\end{eqnarray*}
if we take $\alpha\leq \frac{1}{2}\log\left(\sqrt{2}(\sqrt{2}-1)\right)$. Consequently,
$$
\E\psi_{\theta_0}\geq cn^{\frac{1}{2}-\gamma}=\frac{c\sqrt{n}}{p}=\frac{c}{p^{1-\frac{1}{2\gamma}}}.
$$
\end{proof}

Hence, we obtain the following
\begin{cor}\label{CorollaryExpectationPsi}
If $1\leq p\leq  n$. Then,
$$
\sigma\left\{\theta\in S^{n-1}\,:\,\E\psi_\theta\sim\frac{1}{\sqrt p}\right\}\geq 1-\frac{1}{2^n}.
$$
If $p>n$, then for every $\theta\in S^{n-1}$
$$
\E\psi_\theta\sim\frac{1}{p}\Vert\theta\Vert_1.
$$
\end{cor}
\begin{proof}
The first estimate is a consequence of part c) in Lemma \ref{ExpectationPsi} and Lemma \ref{ExpectationPsiTheta0}. The second estimate is a consequence of part b) in Lemma \ref{ExpectationPsi} and Lemma \ref{ExpectationPsiTheta0}.
\end{proof}
\begin{rmk}
Actually, it can be proved that for any $n\in\N$ and any fixed $\theta\in S^{n-1}$, $\lim_{p\to\infty}p\E\psi_\theta=\Vert\theta\Vert_1$.
\end{rmk}

\begin{lemma}\label{BoundLessTerms}
Let $I\subseteq\{1,\dots, n\}$ be any set of indices and $\theta\in S^{n-1}$. Then,
$$
\frac{\E\left|\sum_{i\in I}|g_i|^{p-1}\signum(g_i)\theta_i\right|}{\E\left|\sum_{i=1}^n|g_i|^{p-1}\signum(g_i)\theta_i\right|}\leq 1
$$
and
$$
\frac{\E\left(\sum_{i\in I}|g_i|^{2p-2}\theta_i^2\right)^\frac{1}{2}}{\E\left(\sum_{i=1}^n|g_i|^{2p-2}\theta_i^2\right)^\frac{1}{2}}\leq C,
$$
where $C$ is an absolute constant.
\end{lemma}
\begin{proof}
By the triangle inequality, we have that
\begin{align*}
2\left|\sum_{i\in I}|g_i|^{p-1}\signum(g_i)\theta_i\right|&\leq\left|\sum_{i\in I}|g_i|^{p-1}\signum(g_i)\theta_i+\sum_{i\in I^c}|g_i|^{p-1}\signum(g_i)\theta_i\right|\cr
&+\left|\sum_{i\in I}|g_i|^{p-1}\signum(g_i)\theta_i-\sum_{i\in I^c}|g_i|^{p-1}\signum(g_i)\theta_i\right|.\cr
\end{align*}
Since the random variables $g_i$ are symmetric, the expected value of the second term equals the expected value of the first term and then
$$
2\E\left|\sum_{i\in I}|g_i|^{p-1}\signum(g_i)\theta_i\right|\leq 2\E\left|\sum_{i=1}^n|g_i|^{p-1}\signum(g_i)\theta_i\right|=2\E\psi,
$$
which proves the first inequality. The second inequality is a consequence of the first one and Khintchine's inequality.
\end{proof}

\section{The variance conjecture on hyperplane projections of
$B_p^n$}\label{proofProjections}
In this section we prove Theorem \ref{projectionsB_p^n}.
\begin{proof}
First of all, notice that, by Proposition 4 in \cite{AB1}, for any $\xi\in
S^{n-1}\cap H$ we have that if $X$ is a random vector uniformly distributed on $P_H(B_p^n)$
$$
\E\,\langle |B_p^n|^{-\frac{1}{n}}X,\xi\rangle^2\sim L_{B_p^n}^2\sim1.
$$
Thus,
$$
\E\,\langle
X,\xi\rangle^2\sim|B_p^n|^{\frac{2}{n}}\sim\frac{1}{n^\frac{2}{p}}
$$
and so
$$
\lambda_X^2\E\,|X|^2\sim n^{1-\frac{4}{p}}.
$$
Now, using the probabilistic representation of $X$ mentioned in Section \ref{Preliminaries}, we have that
\begin{eqnarray*}
\textrm{Var}\,|X|^2&=&\E\,|X|^4-(\E\,|X|^2)^2\cr
&=&\frac{1}{\E \psi}\E
\left|P_H\left(\frac{G}{S}
\right)\right|^4\psi -\frac{1}{(\E \psi)^2}\left(\E
\left|P_H\left(\frac{G}{S}
\right)\right|^2\psi\right)^2\cr
&=&\frac{1}{\E \psi}\E
\left(\left|\frac{G}{S}
\right|^2-\left\langle\frac{G}{S}
 ,\theta\right\rangle^2\right)^2\psi\\
&-&\frac{1}{(\E \psi)^2}
\left(\E\left|\frac{G}{S}
\right|^2\psi-\E\left\langle\frac{G}{S}
 ,\theta\right\rangle^2\psi\right)^2\cr
&\leq&
\frac{1}{\E \psi}
\E\left|\frac{G}{S}
\right|^4\psi-\left(\frac{1}{\E \psi}\E\left|\frac{G}{S}
\right|^2\psi\right)^2\cr
&+&\frac{1}{\E \psi}\E\left\langle\frac{G}{S}
 ,\theta\right\rangle^4\psi +2\frac{1}{(\E
\psi)^2}\E\left|\frac{G}{S}\right|^2\psi
 \,
 \E\left\langle\frac{G}{S} ,\theta\right\rangle^2\psi\cr
&=&\sum_{i=1}^n\left(\frac{1}{\E \psi}\E\frac{g_i^4}{S^4}\psi-\left(\frac{1}{\E
\psi}\E\frac{g_i^2}{S^2}
\psi
\right)^2\right)\cr
&+&\sum_{i\neq
j}\left(\frac{1}{\E \psi}\E\frac{g_i^2g_j^2}{S^4}\psi-\frac{1}{(\E
\psi)^2}\E\frac{g_i^2}{S^2}\psi\E\frac{g_j^2
}{S^2}\psi\right)\cr
&+&\frac{1}{\E \psi}\E\left\langle\frac{G}{S} ,\theta\right\rangle^4\psi
+\frac{2}{(\E \psi)^2}\E\left|\frac{G}{S}\right|^2\psi\,
\E\left\langle\frac{G}{S} ,\theta\right\rangle^2\psi.\cr
\end{eqnarray*}
We are going to bound from above each one of the four summands in the last expression. The upper bound of the first, third, and fourth term will be of the order that would give an absolute constant in the variance conjecture. The estimate we obtain for the second term will be the one that will cause the constant to depend on $p$ if $p\leq n$.
\subsection{Upper bound for the last term}
~\\
By the independence of $\frac{G}{S}$ and $S$ we have that for any $\theta\in S^{n-1}$
\begin{eqnarray*}
\frac{1}{\E \psi}\E\left\langle\frac{G}{S},\theta\right\rangle^2\psi&=&\frac{\E\left\langle\frac{G}{S},\theta\right\rangle^2\left|\sum_{i=1}
^n
\frac{|g_i|^{p-1}}{S^{p-1}}\signum (g_i) \theta_i\right|}{\E \left|\sum_{i=1}
^n
\frac{|g_i|^{p-1}}{S^{p-1}}\signum (g_i) \theta_i\right|}\cr
&=&\frac{\E
S^{p-1}}{\E
S^{p+1}}\frac{\E(\sum_{i=1}^ng_i\theta_i)^2\left|\sum_{i=1}^n|g_i|^{p-1}
\signum(g_i)\theta_i\right|}{\E\left|\sum_{i=1}^n|g_i|^{p-1}
\signum(g_i)\theta_i\right|}.\cr
\end{eqnarray*}
Taking $\varepsilon_1,\dots, \varepsilon_n$
independent Bernoulli
random
variables also independent with respect to the $g_i$'s we have that
\begin{eqnarray*}
&&\frac{\E
S^{p-1}}{\E
S^{p+1}}\frac{\E(\sum_{i=1}^ng_i\theta_i)^2\left|\sum_{i=1}^n|g_i|^{p-1}
\signum(g_i)\theta_i\right|}{\E\left|\sum_{i=1}^n|g_i|^{p-1}
\signum(g_i)\theta_i\right|}\cr
&=&\frac{\E S^{p-1}}{\E
S^{p+1}}\frac{\E_{\varepsilon}\E_g(\sum_{i=1}
^n\varepsilon_ig_i\theta_i)^2\left|\sum_{i=1}^n|g_i|^{p-1}
\signum(g_i)\varepsilon_i\theta_i\right|}{\E_g\left|\sum_{i=1}^n|g_i|^{p-1}
\signum(g_i)\theta_i\right|}\cr
&\leq&\frac{\E S^{p-1}}{\E
S^{p+1}}\frac{\E_g\left(\E_\varepsilon(\sum_{i=1}
^n\varepsilon_ig_i\theta_i)^4\right)^\frac{1}{2}\left(\E_\varepsilon\left|\sum_{
i=1}^n|g_i|^{p-1}
\signum(g_i)\varepsilon_i\theta_i\right|^2\right)^\frac{1}{2}}{\E_g\left|\sum_{
i=1}^n|g_i|^{p-1}
\signum(g_i)\theta_i\right|}.\cr
\end{eqnarray*}
By Khintchine's inequality, Lemma \ref{LemmagS}, Lemma \ref{ExpectationPsi} and Lemma \ref{BoundLessTerms}
\begin{eqnarray*}
\frac{1}{\E\psi}\E\left\langle\frac{G}{S},\theta\right\rangle^2\psi&\leq&C\frac{
\E S^{p-1}}{\E
S^{p+1}}\frac{\E_g\left(\sum_{i=1}^ng_i^2\theta_i^2\right)\left(\sum_{j=1}
^n|g_j|^{2p-2}
\theta_j^2\right)^\frac{1}{2}}{\E_g\left|\sum_{j=1}^n|g_j|^{p-1}
\signum(g_j)\theta_j\right|}\cr
&=&\frac{\E S^{p-1}}{\E
S^{p+1}}\frac{\sum_{i=1}^n\theta_i^2\E_gg_i^2\left(\sum_{j=1}^n|g_j|^{2p-2}
\theta_j^2\right)^\frac{1}{2}}{\E_g\left|\sum_{j=1}^n|g_j|^{p-1}
\signum(g_j)\theta_j\right|}\cr
&\leq&\frac{\E S^{p-1}}{\E
S^{p+1}}\frac{\sum_{i=1}^n\theta_i^2\E_gg_i^2\left(|g_i|^{p-1}|\theta_i|+\left(\sum_{j\neq i}|g_j|^{2p-2}
\theta_j^2\right)^\frac{1}{2}\right)}{\E_g\left|\sum_{j=1}^n|g_j|^{p-1}
\signum(g_j)\theta_j\right|}\cr
&=&\frac{\E S^{p-1}}{\E
S^{p+1}}\frac{\sum_{i=1}^n\theta_i^2\left(\E_g|g_i|^{p+1}|\theta_i|+\E_gg_i^2\E_g\left(\sum_{j\neq i}|g_j|^{2p-2}
\theta_j^2\right)^\frac{1}{2}\right)}{\E_g\left|\sum_{j=1}^n|g_j|^{p-1}
\signum(g_j)\theta_j\right|}\cr
&\leq& \frac{\E S^{p-1}}{\E
S^{p+1}}\sum_{i=1}^n\theta_i^2(C_1|\theta_i|+C_2)\leq C \frac{\E S^{p-1}}{\E
S^{p+1}}.
\end{eqnarray*}
By Lemma \ref{LemmagS} we have $\frac{\E S^{p-1}}{\E
S^{p+1}}\sim\frac{1}{n^\frac{2}{p}}$. Thus
\begin{eqnarray*}
\frac{1}{\E\psi}\E\left\langle\frac{G}{S},\theta\right\rangle^2\psi
&\leq&\frac{C}{n^\frac{2}{p}}.
\end{eqnarray*}
Also, as before,
\begin{eqnarray*}
 \frac{1}{\E\psi}\E\left|\frac{G}{S}
\right|^2\psi&=&\frac{\E S^{p-1}}{\E S^{p+1}}\frac{1}{\E\psi}\sum_{i=1}^n
\E g_i^2\psi\cr
&\leq&\frac{C}{n^{\frac{2}{p}}\E\psi}
\sum_{i=1}^n\left(\E |g_i|^{p+1}|\theta_i|+\E_gg_i^2\E_g\left(\sum_{j\neq i}|g_j|^{2p-2}\theta_j^2\right)^\frac{1}{2}\right)\cr
&\leq&\frac{C}{n^{\frac{2}{p}}}\left(\Vert\theta\Vert_1+ n\right)\cr
&\leq& Cn^{1-\frac{2}{p}}
\end{eqnarray*}
and so
$$
\frac{1}{(\E\psi)^2}\E\left|\frac{G}{S}\right|^2\psi\E\left\langle\frac{G}{S}
,\theta\right\rangle^2\psi\leq Cn^{1-\frac{4}{p}}.$$

\subsection{Upper bound for the first and third term}
~\\
Similarly, by  the independence of $\frac{G}{S}$ and $S$, H\"older's inequality, Khintchine's inequality,  Lemma \ref{LemmagS} and Lemma \ref{BoundLessTerms} we have
\begin{eqnarray*}
&&\frac{1}{\E\psi}\E\left\langle\frac{G}{S},\theta\right\rangle^4\psi=\frac{\E S^{p-1}}{\E S^{p+3}\E\psi}\E_g\E_\varepsilon\left(\sum_{i=1}^ng_i\signum(g_i)\varepsilon_i\theta_i\right)^4\left|\sum_{j=1}^n|g_j|^{p-1}\signum(g_j)\varepsilon_j\theta_j\right|\cr
&\leq&\frac{\E S^{p-1}}{\E S^{p+3}\E\psi}\E_g\left(\E_\varepsilon\left(\sum_{i=1}^ng_i\signum(g_i)\varepsilon_i\theta_i\right)^8\right)^\frac{1}{2}\left(\E_\varepsilon\left|\sum_{j=1}^n|g_j|^{p-1}\signum(g_j)\varepsilon_j\theta_j\right|^2\right)^\frac{1}{2}\cr
&\sim&\frac{\E S^{p-1}}{\E S^{p+3}\E\psi}\E\left(\sum_{i=1}^ng_i^2\theta_i^2\right)^2\left|\sum_{j=1}^n|g_j|^{2p-2}\theta_j^2\right|^\frac{1}{2}\cr
&\leq&\frac{\E S^{p-1}}{\E S^{p+3}\E\psi}\E\sum_{i=1}^ng_i^4\theta_i^2\left|\sum_{j=1}^n|g_j|^{2p-2}\theta_j^2\right|^\frac{1}{2}\cr
&=&\frac{\E S^{p-1}}{\E S^{p+3}\E\psi}\sum_{i=1}^n\theta_i^2\E g_i^4\left|\sum_{j=1}^n|g_j|^{2p-2}\theta_j^2\right|^\frac{1}{2}\cr
&\leq&\frac{\E S^{p-1}}{\E S^{p+3}\E\psi}\sum_{i=1}^n\theta_i^2\E g_i^4\left(|g_i|^{p-1}|\theta_i|+\left(\sum_{j\neq i}|g_j|^{2p-2}\theta_j^2\right)^\frac{1}{2}\right)\cr
&=&\frac{\E S^{p-1}}{\E S^{p+3}\E\psi}\sum_{i=1}^n\theta_i^2\left(\E|g_i|^{p+3}|\theta_i|+\E g_i^4\E\left(\sum_{j\neq i}|g_j|^{2p-2}\theta_j^2\right)^\frac{1}{2}\right)\cr
&\leq&\frac{\E S^{p-1}}{\E S^{p+3}}C\sum_{i=1}^n\theta_i^2\cr
&\leq& C n^{-\frac{4}{p}}
\end{eqnarray*}
since, by Lemma \ref{LemmagS}, $\frac{\E S^{p-1}}{\E S^{p+3}}\sim n^{-\frac{4}{p}}$. This bounds the third term. Besides, this estimate implies the following bound
on the first term:
$$
\sum_{i=1}^n\left(\frac{1}{\E\psi}\E\frac{g_i^4}{S^4}\psi-\left(\frac{1}{\E\psi}
\E\frac{g_i^2}{S^2}\psi
\right)^2\right)\leq\sum_{i=1}^n\frac{1}{\E\psi}\E\frac{g_i^4}{S^4}\psi\leq
Cn^{1-\frac{4}{p}}. $$
\subsection{Upper bound for the second term}
~\\
It remains to bound the second term
$$
\sum_{i\neq
j}\left(\frac{1}{\E \psi}\E\frac{g_i^2g_j^2}{S^4}\psi-\frac{1}{(\E
\psi)^2}\E\frac{g_i^2}{S^2}\psi\E\frac{g_j^2
}{S^2}\psi\right).
$$
For any $i\neq j$ we have
\begin{align*}
\frac{1}{\E \psi}\E\frac{g_i^2g_j^2}{S^4}\psi&-\frac{1}{(\E
\psi)^2}\E\frac{g_i^2}{S^2}\psi\E\frac{g_j^2
}{S^2}\psi=\frac{\E S^{p-1}}{\E S^{p+3}}\frac{\E g_i^2g_j^2\psi}{\E\psi}
-\left(\frac{\E S^{p-1}}{\E S^{p+1}}\right)^2\frac{\E g_i^2\psi\E
g_j^2\psi}{(\E\psi)^2}\\
&= \frac{(\E S^{p-1})^2}{(\E S^{p+1})^2(\E\psi)^2}\left(
\frac{(\E S^{p+1})^2}{\E S^{p-1}\E S^{p+3}}\,\E g_i^2g_j^2\psi\E \psi-
\E g_i^2\psi\E g_j^2\psi
\right).
\end{align*}
By H\"older's inequality, $(\E S^{p+1})^2\leq\E S^{p-1}\E S^{p+3}$. Then, we have
\begin{align*}
\frac{1}{\E \psi}\E\frac{g_i^2g_j^2}{S^4}\psi&-\frac{1}{(\E
\psi)^2}\E\frac{g_i^2}{S^2}\psi\E\frac{g_j^2
}{S^2}\psi\leq
\frac{(\E S^{p-1})^2}{(\E S^{p+1})^2(\E\psi)^2}\left(
\E g_i^2g_j^2\psi\E \psi-
\E g_i^2\psi\E g_j^2\psi
\right).
\end{align*}
Note that if $\{\bar{g}_i\}_{i=1}^n$ are independent copies of $g$, independent
of $\{g_i\}_{i=1}^n$, and $\bar{\psi}=\left|\sum_{i=1}^n
|\bar{g}_i|^{p-1}\signum (\bar{g}_i) \theta_i\right|$, we have that
\begin{eqnarray*}
\E g_i^2g_j^2\psi\E \psi-
\E g_i^2\psi\E
g_j^2\psi&=&\E_{g\otimes\bar{g}}g_i^2(g_j^2-\bar{g}_j^2)\psi\bar{\psi}\cr
&=&\E_{g\otimes\bar{g}}\bar{g}_i^2(\bar{g}_j^2-g_j^2)\psi\bar{\psi}.\cr
\end{eqnarray*}
Thus,
$$
\E g_i^2g_j^2\psi\E \psi-\E g_i^2\psi\E
g_j^2\psi=\frac{1}{2}\E_{g\otimes\bar{g}}(g_i^2-\bar{g}_i^2)(g_j^2-\bar{g}
_j^2)\psi\bar{\psi}
$$
and so,
\begin{align*}
&\sum_{i\neq j}\frac{1}{\E \psi}\E\frac{g_i^2g_j^2}{S^4}\psi-\frac{1}{(\E
\psi)^2}\E\frac{g_i^2}{S^2}\psi\E\frac{g_j^2
}{S^2}\psi\leq\frac{(\E S^{p-1})^2}{2(\E S^{p+1})^2(\E\psi)^2}\E\psi\bar{\psi}\sum_{i\neq j}(g_i^2-\bar{g}_i^2)(g_j^2-\bar{g}
_j^2)\cr
&\leq \frac{(\E S^{p-1})^2}{2(\E S^{p+1})^2(\E\psi)^2}\E\psi\bar{\psi}\left(\sum_{i=1}^n(g_i^2-\bar{g}_i^2)\right)^2\sim\frac{n^{-\frac{4}{p}}}{(\E\psi)^2}\E\psi\bar{\psi}\left(\sum_{i=1}^n(g_i^2-\bar{g}_i^2)\right)^2.\cr
\end{align*}
Now, for any $\alpha\geq 1$, this is bounded by
$$
\leq n^{-\frac{4}{p}}\left(\frac{\E\psi\bar{\psi}\left(\sum_{i=1}^n(g_i^2-\bar{g}_i^2)\right)^{2\alpha}}{(\E\psi)^2}\right)^\frac{1}{\alpha}\leq n^{-\frac{4}{p}}\frac{\left(\E\psi^2\right)^\frac{1}{\alpha}}{\left(\E\psi\right)^\frac{2}{\alpha}}\left(\E\left(\sum_{i=1}^n(g_i^2-\bar{g}_i^2)\right)^{4\alpha}\right)^\frac{1}{2\alpha}.
$$
By Lemma \ref{ExpectationPsi}, $\frac{\left(\E\psi^2\right)^\frac{1}{\alpha}}{\left(\E\psi\right)^\frac{2}{\alpha}}\leq Cp^\frac{1}{\alpha}$ and, taking $\alpha\sim \log p$ we have by Lemma \ref{MomentsSumsSymmetric} that
$$
\left(\E\left(\sum_{i=1}^n(g_i^2-\bar{g}_i^2)\right)^{4\alpha}\right)^\frac{1}{2\alpha}\sim n\log p
$$
and so
$$
\sum_{i\neq
j}\left(\frac{1}{\E \psi}\E\frac{g_i^2g_j^2}{S^4}\psi-\frac{1}{(\E
\psi)^2}\E\frac{g_i^2}{S^2}\psi\E\frac{g_j^2
}{S^2}\psi\right)\leq Cn^{1-\frac{4}{p}}\log(1+p).
$$
Besides, by  Corollary \ref{CorollaryExpectationPsi}, if $1\leq p\leq n$, $\frac{\left(\E\psi^2\right)^\frac{1}{\alpha}}{\left(\E\psi\right)^\frac{2}{\alpha}}\leq C$ for a set of directions of measure greater than $1-\frac{1}{2^n}$. Taking $\alpha=2$ when $\theta$ belongs to this set we obtain
$$
\sum_{i\neq
j}\left(\frac{1}{\E \psi}\E\frac{g_i^2g_j^2}{S^4}\psi-\frac{1}{(\E
\psi)^2}\E\frac{g_i^2}{S^2}\psi\E\frac{g_j^2
}{S^2}\psi\right)\leq Cn^{1-\frac{4}{p}}.
$$
This finishes the proof  in the case $p\leq n$.

In the case that $p>n$ we take into account that, since the random variables $g_i$ are independent and identically distributed
\begin{eqnarray*}
\E\psi\bar{\psi}\left(\sum_{i=1}^n(g_i^2-\bar{g}_i^2)\right)^2&=&\E\psi\bar{\psi}\left(\sum_{i=1}^n(g_i^2-\E g_i^2)+\sum_{i=1}^n(\E\bar{g}_i^2-\bar{g}_i^2)\right)^2\cr
&\leq&\sqrt2\E\psi\bar{\psi}\left[\left(\sum_{i=1}^n(g_i^2-\E g_i^2)\right)^2+\left(\sum_{i=1}^n(\bar{g}_i^2-\E\bar{g}_i^2)\right)^2\right]\cr
&=&2\sqrt2\E\psi\E\psi\left(\sum_{i=1}^n(g_i^2-\E g_i^2)\right)^2\cr
&\leq&2\sqrt2\E\psi\E\sum_{j=1}^n|g_j|^{p-1}|\theta_j|\left(\sum_{i=1}^n(g_i^2-\E g_i^2)\right)^2\cr
&=&2\sqrt2\Vert\theta\Vert_1\E\psi\E|g_1|^{p-1}\left(\sum_{i=1}^n(g_i^2-\E g_i^2)\right)^2\cr
&=&2\sqrt2\Vert\theta\Vert_1\E\psi\E|g_1|^{p-1}\sum_{i=1}^n(g_i^2-\E g_i^2)^2\cr
&\leq&\frac{Cn\Vert\theta\Vert_1\E\psi}{p}.
\end{eqnarray*}
Since by part b) in Lemma \ref{ExpectationPsi}
$$
\E\psi_\theta\geq\frac{c_1}{\sqrt n}\E\psi_{\theta_0}\Vert\theta\Vert_1,
$$
we have that
$$
\frac{n^{-\frac{4}{p}}}{(\E\psi)^2}\E\psi\bar{\psi}\left(\sum_{i=1}^n(g_i^2-\bar{g}_i^2)\right)^2\leq C\frac{\sqrt n}{p\E\psi_{\theta_0}}n^{1-\frac{4}{p}}
$$
and, since $p\geq n$, by Lemma \ref{ExpectationPsiTheta0} $\E\psi_{\theta_0}\sim\frac{\sqrt n}{p}$ and we obtain the result.
\end{proof}

\section{Hyperplane projections of isotropic random vectors and Steiner
symmetrization}\label{proofSteiner}

In this section we will show how the variance conjecture for an isotropic
log-concave random vector relates to the variance conjecture for its hyperplane projections or for
its Steiner symmetrizations (when the vector is uniformly distributed on an isotropic body).
\begin{proposition}
Let $\mu$ be a log-concave probability on $\R^n$ and $X$ a random vector distributed according to $\mu$. Then for any linear subspace
$E$
$$
\left|\sqrt{\textrm{Var}|X|^2}-\sqrt{\textrm{Var}|P_EX|^2}\right|\leq\sqrt{
\textrm{Var}|P_{E^\perp}X|^2}.
$$
\end{proposition}
\begin{proof}
For any linear subspace $E$,
$$
|X|^2=|P_{E}(X)|^2+|P_{E^\perp}(X)|^2.
$$
Thus,
\begin{eqnarray*}
\textrm{Var}|X|^2&=&\textrm{Var}|P_EX|^2+\textrm{Var}|P_{E^\perp}X|^2\cr
&+&2\left(\E|P_EX|^2|P_{E^\perp}X|^2-\E|P_EX|^2\E|P_{E^\perp}X|^2\right)\cr
&=&\textrm{Var}|P_EX|^2+\textrm{Var}|P_{E^\perp}X|^2\cr
&+&2\left(\E|P_EX|^2(|X|^2-|P_EX|^2)-\E|P_EX|^2\E(|X|^2-|P_EX|^2)\right)\cr
&=&\textrm{Var}|P_EX|^2+\textrm{Var}|P_{E^\perp}X|^2\cr
&+&2\E|P_EX|^2(|X|^2-\E|X|^2)-2\textrm{Var}|P_EX|^2\cr
&=&\textrm{Var}|P_{E^\perp}X|^2-\textrm{Var}|P_EX|^2\cr
&+&2\E(|P_EX|^2-\E|P_EX|^2)(|X|^2-\E|X|^2)\cr
&\leq&\textrm{Var}|P_{E^\perp}X|^2-\textrm{Var}|P_EX|^2\cr
&+&2\sqrt{\textrm{Var}|P_EX|^2}\sqrt{\textrm{Var}|X|^2}.\cr
\end{eqnarray*}
Consequently,
$$
\textrm{Var}|X|^2-2\sqrt{\textrm{Var}|P_EX|^2}\sqrt{\textrm{Var}|X|^2}-\textrm{
Var}|P_{E^\perp}X|^2+\textrm{Var}|P_EX|^2\leq 0.
$$
Since the roots of the polyomial
$$p(x)=x^2-2\sqrt{\textrm{Var}|P_EX|^2}x-\textrm{Var}|P_{E^\perp}X|^2+\textrm{
Var}|P_EX|^2$$ are
$$
\sqrt{\textrm{Var}|P_EX|^2}\pm\sqrt{\textrm{Var}|P_{E^\perp}X|^2},
$$
we obtain the result.
\end{proof}

As a consequence, we have the following
\begin{thm}\label{projections}
Let $X$ be an isotropic log-concave random vector. Then the following
are equivalent
\begin{itemize}
\item There exists a constant $C_1$ such that $$\textrm{Var}\, |X|^2\leq
C_1n.$$
\item There exists a constant $C_2$ such that $$\textrm{Var}\, |P_EX|^2\leq
C_2(n-1)$$  for some hyperplane $E$.
\item There exists a constant $C_3$ such that $$\textrm{Var}\, |P_EX|^2\leq
C_3(n-1)$$  for every hyperplane $E$,
\end{itemize}
where
$$
C_2\leq C_3\leq 4\left(C_1+\frac{C}{n}\right) \textrm{ and }
C_1\leq2\left(C_2+\frac{C}{n}\right),
$$
with $C$ an absolute constant.
\end{thm}
\begin{proof}
Let $E=\theta^\perp$ be a hyperplane and $X$ an isotropic log-concave random vector. Since $X$ is isotropic, also $P_EX$ is isotropic. Thus, if
$X$ verifies the variance conjecture with constant $C_1$ then, for every
hyperplane $E=\theta^\perp$
\begin{eqnarray*}
\sqrt{\textrm{Var}|P_EX|^2}&\leq&\sqrt{\textrm{Var}|X|^2}+\sqrt{\textrm{Var}
\langle X,\theta\rangle^2}\cr
&\leq&\sqrt{C_1n}+\sqrt{\E\langle X,\theta\rangle^4}\cr
&\leq&\sqrt 2\sqrt{C_1n+\E\langle X,\theta\rangle^4}.\cr
\end{eqnarray*}
By Borell's inequality
$$
\sqrt{\textrm{Var}|P_EX|^2}\leq\sqrt2\sqrt{C_1n+C^\prime}=\sqrt2\sqrt{C_1+\frac{
C}{n}}\sqrt n.
$$
Thus, there exists an absolute constant $C$ such that
$$
\textrm{Var}|P_EX|^2\leq
2\left(C_1+\frac{C}{n}\right)n\leq4\left(C_1+\frac{C}{n}\right)(n-1).
$$
In the same way, if there exists a hyperplane $E=\theta^\perp$ such that
$\textrm{Var}\,|P_EX|^2\leq C_2(n-1)$ , then
\begin{eqnarray*}
\sqrt{\textrm{Var}|X|^2}&\leq&\sqrt{\textrm{Var}|P_EX|^2}+\sqrt{\textrm{Var}
\langle X,\theta\rangle^2}\cr
&\leq&\sqrt{C_2(n-1)}+\sqrt{\E\langle X,\theta\rangle^4}\cr
&\leq&\sqrt 2\sqrt{C_2(n-1)+\E\langle X,\theta\rangle^4}\cr
\end{eqnarray*}
and, by Borell's inequality,
$$
\sqrt{\textrm{Var}|X|^2}\leq\sqrt2\sqrt{C_2n+C}=\sqrt2\sqrt{C_2+\frac{C}{n}}
\sqrt n.
$$
Thus, there exists an absolute constant $C$ such that
$$
\textrm{Var}|X|^2\leq 2\left(C_2+\frac{C}{n}\right)n.
$$
\end{proof}

Now we will prove Theorem \ref{SteinerSymmetrization}. It will be a consequence
of the following
\begin{proposition}
Let $K$ be an isotropic convex body, $\theta\in S^{n-1}$ and $S_\theta(K)$ its
Steiner symmetrization with respect to the hyperplane $H=\theta^\perp$. Let $Y$
be a random vector uniformly distributed on $S_\theta(K)$ and $X$ a random
vector uniformly distributed on $K$. Then there exists an absolute constant $C$
such that
$$
\left|\textrm{Var}|Y|^2-\textrm{Var}|X|^2\right|\leq CnL_K^4.
$$
\end{proposition}
\begin{proof}
Without loss of generality we can assume that $\theta=e_n$. We have that
\begin{eqnarray*}
\textrm{Var}|Y|^2&=&\sum_{i=1}^n\left(\E\langle Y,e_i\rangle^4-(\E\langle
Y,e_i\rangle^2)^2\right)\cr
&+&\sum_{i\neq j}\left(\E\langle Y,e_i\rangle^2\langle Y,e_j\rangle^2-\E\langle
Y,e_i\rangle^2\E\langle Y,e_j\rangle^2\right).
\end{eqnarray*}
Notice that if $i\neq n$
\begin{eqnarray*}
\E\langle Y,e_i\rangle^4&=&\int_{P_H(K)}\langle
y,e_i\rangle^4|S_\theta(K)\cap(y+\langle e_n\rangle)|dy\cr
&=&\int_{P_H(K)}\langle y,e_i\rangle^4|K\cap(y+\langle e_n\rangle)|dy
=\E\langle X,e_i\rangle^4.
\end{eqnarray*}
If $i=n$ and for every $y\in P_H(K)$ we have that $K\cap (y+\langle e_n\rangle)$
is the segment $[a(y),b(y)]e_n$, which has length $2l(y)$
\begin{eqnarray*}
\E\langle
Y,e_n\rangle^4&=&\int_{P_H(K)}\int_{-l(y)}^{l(y)}t^4dtdy\leq\int_{P_H(K)}\int_{
a(y)}^{b(y)}t^4dtdy\cr
&=&\E\langle X,e_n\rangle^4.
\end{eqnarray*}
In the same way, if $i\neq n$
$$
\E\langle Y,e_i\rangle^2=\E\langle X,e_i\rangle^2
$$
and if $i=n$
$$
\E\langle Y,e_n\rangle^2\leq\E\langle X,e_n\rangle^2.
$$
Besides, if $i,j\neq n$
$$
\E\langle Y,e_i\rangle^2\langle Y,e_j\rangle^2=\E\langle X,e_i\rangle^2\langle
X,e_j\rangle^2
$$
and if $i\neq n$
\begin{eqnarray*}
\E\langle Y,e_i\rangle^2\langle Y,e_n\rangle^2&=&\int_{P_H(K)}\langle
y,e_i\rangle^2\int_{-l(y)}^{l(y)}t^2dtdy\leq \int_{P_H(K)}\langle
y,e_i\rangle^2\int_{a(y)}^{b(y)}t^2dtdy\cr
&=&\E\langle X,e_i\rangle^2\langle X,e_n\rangle^2.
\end{eqnarray*}
Thus,
\begin{eqnarray*}
\textrm{Var}|Y|^2=\textrm{Var}|X|^2&+&\E\langle Y,e_n\rangle^4-\E\langle
X,e_n\rangle^4\cr
&+&(\E\langle X,e_n\rangle^2)^2-(\E\langle Y,e_i\rangle^2)^2\cr
&+&2\sum_{i=1}^{n-1}\E\langle Y,e_i\rangle^2\langle Y,e_n\rangle^2-\E\langle
X,e_i\rangle^2\langle X,e_n\rangle^2\cr
&+&2\sum_{i=1}^{n-1}\E\langle X,e_i\rangle^2(\E\langle X,e_n\rangle^2-\E\langle
Y,e_n\rangle^2).\cr
\end{eqnarray*}
Consequently
\begin{eqnarray*}
\textrm{Var}|Y|^2&\leq&\textrm{Var}|X|^2\cr
&+&(\E\langle X,e_n\rangle^2)^2-(\E\langle Y,e_i\rangle^2)^2\cr
&+&2\sum_{i=1}^{n-1}\E\langle X,e_i\rangle^2(\E\langle X,e_n\rangle^2-\E\langle
Y,e_n\rangle^2).\cr
\end{eqnarray*}
Now, if $K$ is isotropic
\begin{eqnarray*}
\textrm{Var}|Y|^2&\leq&\textrm{Var}|X|^2\cr
&+&L_K^4-(\E\langle Y,e_i\rangle^2)^2\cr
&+&2(n-1)L_K^2(L_K^2-\E\langle Y,e_n\rangle^2)\cr
&\leq&\textrm{Var}|X|^2´+(2n-1)L_K^4.
\end{eqnarray*}

On the other hand, by H\"older's inequality and Borell's lemma
\begin{eqnarray*}
\textrm{Var}|Y|^2&\geq&\textrm{Var}|X|^2-\E\langle X,e_n\rangle^4\cr
&-&2\sum_{i=1}^{n-1}\E\langle X,e_i\rangle^2\langle X,e_n\rangle^2\cr
&\geq&\textrm{Var}|X|^2-\E\langle X,e_n\rangle^4\cr
&-&2\sum_{i=1}^{n-1}\E(\langle X,e_i\rangle^4)^\frac{1}{2}\E(\langle
X,e_n\rangle^4)^\frac{1}{2}\cr
&\geq&\textrm{Var}|X|^2-C(\E\langle X,e_n\rangle^2)^2\cr
&-&C\sum_{i=1}^{n-1}\E\langle X,e_i\rangle^2\E\langle X,e_n\rangle^2\cr
\end{eqnarray*}
Thus, if $K$ is isotropic
$$
\textrm{Var}|Y|^2\geq\textrm{Var}|X|^2-CnL_K^4.
$$
\end{proof}
As a consequence, we have Theorem \ref{SteinerSymmetrization}:
\begin{proof}[Proof of Theorem \ref{SteinerSymmetrization}]
Let $K$ be an isotropic convex body and let $Y_\theta$ be a random vector on
$S_\theta(K)$. Then $$\lambda_{Y_\theta}^2=L_K^2$$ and
$$
\E|Y_\theta|^2=(n-1)L_K^2+\E\langle Y_\theta,\theta\rangle^2.
$$
Thus $(n-1)L_K^2\leq\E|Y_\theta|^2\leq nL_K^2$ and so, by the previous
proposition, if $X$ verifies the variance conjecture with constant $C_1$ then
for any $\theta\in S^{n-1}$
$$
\textrm{Var}\,|Y_\theta|^2\leq \textrm{Var}\,|X|^2+CnL_K^4\leq
(C_1+C)nL_K^4\leq2(C_1+C)\lambda_{Y_\theta}^2\E|Y|^2
$$
and if for some $\theta\in S^{n-1}$ $Y_\theta$ verifies the variance
conjecture with constant $C_2$ then
$$
\textrm{Var}\,|X|^2\leq \textrm{Var}\,|Y_\theta|^2+CnL_K^4\leq
(C_2+C)nL_K^4=(C_2+C)\lambda_{X}^2\E|X|^2.
$$
\end{proof}

\section{Acknowledgements}
We would like to thank the anonymous referees for several useful comments that helped us to shorten the proofs of some lemmas and improve the presentation of the paper.

\end{document}